\documentclass[reqno]{amsart}
\usepackage{amsmath,amsthm,amscd,amssymb,amsfonts, amsbsy}
\usepackage{latexsym}
\usepackage{color}
\usepackage{enumerate}
\usepackage{pxfonts}
\usepackage{verbatim}

\theoremstyle{plain}
\newtheorem{theorem}[equation]{Theorem}
\newtheorem{lemma}[equation]{Lemma}

\theoremstyle{definition}

\theoremstyle{remark}
\newtheorem{remark}[equation]{Remark}

\newcommand{\dv}{\operatorname{div}}

\newcommand{\diam}{\operatorname{diam}}

\newcommand{\tr}{\operatorname{tr}}

\numberwithin{equation}{section}

\newcommand{\bR}{\mathbb{R}}

\providecommand{\set}[1]{\{#1\}}

\providecommand{\abs}[1]{\lvert#1\rvert}
\providecommand{\Abs}[1]{\left\lvert#1\right\rvert}

\providecommand{\norm}[1]{\lVert#1\rVert}

\begin{document}
\title[Green function for non-divergence elliptic equation]
{Green's function for second order elliptic equations in non-divergence form}

\author[S. Hwang]{Sukjung Hwang}
\address[S. Hwang]{Department of Mathematics, Yonsei University, 50 Yonsei-ro, Seodaemun-gu, Seoul 03722, Republic of Korea}
\email{sukjung\_hwang@yonsei.ac.kr}
\thanks{S. Hwang is partially supported by the Yonsei University Research Fund (Post Doc. Researcher Supporting Program) of 2017 (project no.: 2017-12-0031).}

\author[S. Kim]{Seick Kim}
\address[S. Kim]{Department of Mathematics, Yonsei University, 50 Yonsei-ro, Seodaemun-gu, Seoul 03722, Republic of Korea}
\email{kimseick@yonsei.ac.kr}
\thanks{S. Kim is partially supported by National Research Foundation of Korea (NRF) Grant No. NRF-2016R1D1A1B03931680 and No. NRF-20151009350.}



\begin{abstract}
We construct the Green function for second order elliptic equations in non-divergence form when the mean oscillations of the coefficients satisfy the Dini condition and the domain has $C^{1,1}$ boundary.
We also obtain pointwise bounds for the Green functions and its derivatives.
\end{abstract}
\maketitle

\section{Introduction and main results}
Let $\Omega$ be a bounded $C^{1,1}$ domain (open connected set) in $\bR^n$ with $n \ge 3$.
We consider a second-order elliptic operator $L$ in non-divergence form
\begin{equation}							\label{master-nd}
L u= \sum_{i,j=1}^n a^{ij}(x) D_{ij} u.
\end{equation}
We assume that the coefficient $\mathbf{A}:=(a^{ij})$  is an $n \times n$ real symmetric matrix-valued function defined on $\bR^n$, which satisfies the uniform ellipticity condition
\begin{equation}					\label{ellipticity-nd}
\lambda \abs{\xi}^2 \le \sum_{i,j=1}^n a^{ij}(x) \xi^i \xi^j \le \Lambda \abs{\xi}^2,\quad \forall \, \xi=(\xi^1,\ldots, \xi^n) \in \bR^n,\quad \forall\, x \in \bR^n
\end{equation}
for some constants $0< \lambda \le \Lambda$.

In this article, we are concerned with construction and pointwise estimates for the Green's function $G(x,y)$ of the non-divergent operator $L$ \eqref{master-nd} in $\Omega$.
Unlike the Green's function for uniformly elliptic operators in divergence form, the Green's function for non-divergent elliptic operators does not necessarily enjoy the usual pointwise bound
\begin{equation}				\label{bound1}
G(x,y) \le c \abs{x-y}^{2-n}
\end{equation}
even in the case when the coefficient $\mathbf A$ is uniformly continuous; see \cite{Bauman84}.
On the other hand, in the case when the coefficient $\mathbf A$ is H\"older continuous, then it is well known that the Green's function satisfies the pointwise bound \eqref{bound1}; see e.g., \cite{Friedman} for the construction of fundamental solutions of parabolic operators by the parametrix method.
In this perspective, it is an interesting question to ask what is the minimal regularity condition to ensure the Green's function to have the pointwise bound \eqref{bound1}.
We shall show that if the coefficient $\mathbf A$ is of \emph{Dini mean oscillation}, then the Green's function exists and satisfies the pointwise bound \eqref{bound1}.
We shall say that a function is of Dini mean oscillation if its mean oscillation satisfies the Dini condition.
Here, we briefly describe the role of this Dini mean oscillation condition because it will be used somewhat implicitly in the paper.
First, it will imply that the coefficient $\mathbf A$ is uniformly continuous so that the Calder\'on-Zygmund $L^p$ theory can be applied.
Also, it will provide us a local $L^\infty$ estimate for the solutions of the adjoint equation $L^\ast u=0$ as appears in \eqref{eq1720m}, which is one of the main results of the very recent papers by the second author and collaborators \cite{DK17, DEK17}.
This $L^\infty$ estimate is crucial for the pointwise bound \eqref{bound1} and the uniform continuity of the coefficient $\mathbf A$ alone is not enough to produce such an estimate.
Below is a more precise formulation of Dini mean oscillation condition.

For $x\in \bR^n$ and $r>0$, we denote by $B(x,r)$ the Euclidean ball with radius $r$ centered at $x$, and denote
\[
\Omega(x,r):=\Omega \cap B(x,r).
\]
We shall say that a function $g: \Omega \to \bR $ is of Dini mean oscillation
in $\Omega$ if the mean oscillation function $\omega_g: \bR_+ \to \bR$ defined by
\begin{equation*}					
\omega_g(r):=\sup_{x\in \overline{\Omega}} \fint_{\Omega(x,r)} \,\abs{g(y)-\bar {g}_{\Omega(x,r)}}\,dy \quad \text{ where } \;\bar g_{\Omega(x,r)} :=\fint_{\Omega(x,r)} g\;
\end{equation*}
satisfies the Dini condition; i.e.,
\[
\int_0^1 \frac{\omega_g(t)}t \,dt <+\infty.
\]
It is clear that if $g$ is Dini continuous, then $g$ is of Dini mean oscillation.
However, the Dini mean oscillation condition is strictly weaker than the Dini continuity; see \cite{DK17} for an example.
Also if $g$ is of Dini mean oscillation, then $g$ is uniformly continuous in $\Omega$ with its modulus of continuity controlled by $\omega_g$; see Appendix.

The formal adjoint operator $L^\ast$ is given by
\begin{equation}\label{adj_L}
L^\ast u = \sum_{i,j=1}^n D_{ij} (a^{ij}(x) u).
\end{equation}
We need to consider the boundary value problem of the form
\begin{equation}				\label{adj_eq}
L^\ast u = \dv^2 \mathbf{g} + f\;\text{ in }\;\Omega,\quad u=\frac{\mathbf{g} \nu\cdot \nu}{\mathbf{A}\nu\cdot \nu}\;\text{ on }\;\partial \Omega,
\end{equation}
where $\mathbf{g}=(g^{ij})$ is an $n \times n$ matrix-valued function,
\[
\textstyle \dv^2 \mathbf{g}=\sum_{i,j=1}^n D_{ij}g^{ij},
\]
and $\nu$ is the unit exterior normal vector of $\partial\Omega$.
For $\mathbf{g} \in L^p(\Omega)$ and $f \in L^p(\Omega)$, where $1<p<\infty$ and $\frac{1}{p}+ \frac{1}{p'}=1$, we say that $u$ in $L^p(\Omega)$ is an adjoint solution of \eqref{adj_eq} if $u$ satisfies
\begin{equation}			\label{def01}
\int_\Omega u Lv = \int_\Omega \tr(\mathbf{g} D^2v) + \int_\Omega f v
\end{equation}
for any $v$ in $W^{2,p'}(\Omega) \cap W^{1,p'}_0(\Omega)$.
The following lemma is quoted from \cite[Lemma~2]{EM2016}.
\begin{lemma}				\label{lem01}
Let $1<p<\infty$ and assume that $\mathbf{g} \in L^p(\Omega)$ and $f \in L^p(\Omega)$.
Then there exists a unique adjoint solution $u$ in $L^p(\Omega)$.
Moreover, the following estimates holds
\[
\norm{u}_{L^p(\Omega)} \le C \left[ \norm{\mathbf g}_{L^p(\Omega)} + \norm{f}_{L^p(\Omega)} \right],
\]
where a constant $C$ depends on $\Omega$, $p$, $n$, $\lambda$, $\Lambda$, and the continuity of $\mathbf A$.
\end{lemma}
We clarify that ``the continuity'' of $\mathbf{A}$ in Lemma~\ref{lem01} specifically means ``the modulus of continuity'' of $\mathbf{A}$, which is clear from the context in \cite{EM2016}.
By the modulus of continuity of $\mathbf A$, we mean the function $\varrho_{\mathbf A}$ defined by
\[
\varrho_{\mathbf A}(t):=\sup \left\{ \, \abs{ \mathbf{A}(x)- \mathbf{A}(y)}: x,y \in \Omega, \, \abs{x-y} \le t \,\right\},\quad \forall t \ge 0.
\]
Therefore, in the case when coefficient $\mathbf A$ is of Dini mean oscillation, the constant $C$ in Lemma~\ref{lem01} depends only on $\Omega$, $p$, $n$, $\lambda$, $\Lambda$, and $\omega_{\mathbf A}$.

It is also known that if $f \in L^p(\Omega)$ with $p>\frac{n}{2}$, then the adjoint solution of the problem
\begin{equation}				\label{eq20.40th}
L^\ast u = f\;\text{ in }\;\Omega,\quad u=0 \;\text{ on }\;\Omega,
\end{equation}
is uniformly continuous in $\Omega$; see Theorem~1.8 in \cite{DEK17}.

We say $\partial \Omega$ is $C^{k,{\rm Dini}}$ if for each point $x_0 \in \partial\Omega$, there exist a constant $r>0$ independent of $x_0$ and a $C^{k, {\rm Dini}}$ function (i.e., $C^k$ function whose $k$th derivatives are uniformly Dini continuous) $\gamma: \bR^{n-1} \to \bR$ such that (upon relabeling and reorienting the coordinates axes if necessary) in a new coordinate system $(x',x^n)=(x^1,\ldots,x^{n-1},x^n)$, $x_0$ becomes the origin and
\[
\Omega \cap B(0, r)=\set{ x \in B(0, r) : x^n > \gamma(x^1, \ldots, x^{n-1})},\quad  \gamma(0')=0.
\]

A few remarks are in order before we state our main theorem.
There are many papers in the literature dealing with the existence and estimates of Green's functions or fundamental solutions of non-divergence form elliptic operators with measurable or continuous coefficients.
To our best knowledge, the first author who considered Green's function for non-divergence form elliptic operators with measurable coefficients is Bauman \cite{Bauman84b, Bauman85}, who introduced the concept of normalized adjoint solutions; see also Fabes et al. \cite{FGMS88}.
Fabes and Stroock \cite{FS84} established $L^p$-integrablity of Green's functions for non-divergence form elliptic operators with measurable coefficients.
Krylov \cite{Krylov92} showed that the weak uniqueness property holds for solutions of non-divergence form elliptic equations in $\Omega$ if and only if there is a unique Green's function in $\Omega$.
Escauriaza \cite{Esc2000} established bounds for fundamental solution for non-divergence form elliptic operators in terms of nonnegative adjoint solution.
We would like to thank Luis Escauriaza for bringing our attention to these results in the literature.

Now, we state our main theorem.

\begin{theorem}					\label{thm-main-n}
Let $\Omega$ be a bounded $C^{1, 1}$ domain in $\bR^n$ with $n\ge 3$.
Assume the coefficient $\mathbf{A}=(a^{ij})$ of the non-divergent operator $L$ in \eqref{master-nd} satisfies the uniform ellipticity condition \eqref{ellipticity-nd} and is of Dini mean oscillation in $\Omega$.
Then, there exists a Green's function $G(x,y)\ ( \text{for any } \ x, y \in \Omega, \ x\neq y)$ of the operator $L$ in $\Omega$ and it is unique in the following sense:
if $u$ is the unique adjoint solution of the problem \eqref{eq20.40th}, where $f \in L^p(\Omega)$ with $p>\frac{n}{2}$, then $u$ is represented by
\begin{equation}				\label{eq1747m}
u(y)=\int_\Omega G(x,y) f(x)\,dx.
\end{equation}
The Green function $G(x,y)$ satisfies the following pointwise estimates:
\begin{align}				
					\label{eq0529af}
\abs{G(x,y)} &\le C \abs{x-y}^{2-n}, \\
					\label{eq0530af}
\abs{D_x G(x,y)} &\le C \abs{x-y}^{1-n}  ,
\end{align}
where $C=C(n, \lambda, \Lambda, \Omega, \omega_{\mathbf A})$.
Moreover, if the boundary $\partial \Omega$ is $C^{2, {\rm Dini}}$, then we have
\begin{equation}		\label{eq0531af}
\abs{D_x^2 G(x,y)} \le C \abs{x-y}^{-n},					
\end{equation}
where $C=C(n, \lambda, \Lambda, \Omega, \omega_{\mathbf A})$.
\end{theorem}
\begin{remark}
In the proof of Theorem~\ref{thm-main-n}, we will construct the Green's function $G^\ast(x,y)$ for the adjoint operator $L^\ast$ as a by-product.
It is characterized as follows:
for $q>\frac{n}{2}$ and $f\in L^{q}(\Omega)$,
if $v \in W^{2,q}(\Omega)\cap W^{1,q}_0(\Omega)$ is the strong solution of
\[
Lv=f \;\text{ in }\;\Omega,\quad v=0\;\text{ on }\;\partial \Omega,
\]
then, we have the representation formula
\[
v(y)=\int_\Omega G^\ast(x,y)f(x)\,dx.
\]
Also, in the proof of Theorem~\ref{thm-main-n}, we shall show that
\begin{equation}		\label{symmetry}
G(x,y)=G^\ast(y,x),\quad \forall\, x, y \in \Omega, \quad x \neq y.
\end{equation}
Finally, by the maximum principle, it is clear that $G(x,y) \ge 0$.
\end{remark}
\section{Proof of Theorem~\ref{thm-main-n}}
\subsection{Construction of Green's function}\label{SS:Green_L}

To construct Green's function, we follow the scheme of \cite{HK07}, which in turn is based on \cite{GW82}.
For $y \in \Omega$ and $\epsilon>0$, let $v_\epsilon = v_{\epsilon; y} \in W^{2,2}(\Omega) \cap W^{1,2}_0(\Omega)$ be a unique strong solution of the problem
\begin{equation}			\label{eq1655th}
 Lv=\frac{1}{\abs{\Omega(y,\epsilon)}} \, \chi_{\Omega(y,\epsilon)}\;\text{ in }\;\Omega,\quad v=0\;\text{ on }\;\partial\Omega.
\end{equation}
Note that $\mathbf A$ is uniformly continuous in $\Omega$ with its modulus of continuity controlled by $\omega_{\mathbf A}$.
Therefore, the unique solvability of the problem \eqref{eq1655th} is a consequence of standard $L^{p}$ theory; see e.g., Chapter~9 of \cite{GT}.
Also, by the same theory, we see that $v_\epsilon$  belong to $W^{2,p}(\Omega)$ for any $p \in (1,\infty)$ and we have an estimate
\begin{equation}			\label{eq1659th}
\norm{v_\epsilon}_{W^{2,p}(\Omega)} \le C \,\epsilon^{-n+\frac{n}{p}},\quad  \forall\, \epsilon \in (0, \diam \Omega),
\end{equation}
where $C=C(n, \lambda, \Lambda, p, \Omega, \omega_{\mathbf A})$.
In particular, we see that $v_\epsilon$ is continuous in $\Omega$.
Next, for $f \in C^\infty_c(\Omega)$, consider the adjoint problem
\[
L^\ast u = f\;\text{ in }\;\Omega,\quad u=0 \;\text{ on }\;\partial \Omega.
\]
By Lemma~\ref{lem01}, there exists a unique adjoint solution $u$ in $L^2(\Omega)$, and  by \eqref{def01}, we have
\begin{equation}				\label{eq20.56th}
\fint_{\Omega(y,\epsilon)} u = \int_\Omega f v_\epsilon.
\end{equation}

Let $w$ be the solution of the Dirichlet problem
\begin{equation}				\label{eq21.08th}
\Delta w = f\;\text{ in }\;\Omega,\quad w=0 \;\text{ on }\;\partial\Omega.
\end{equation}
By the standard $L^p$ theory and Sobolev's inequality, for $1<p<\frac{n}{2}$, we have
\begin{equation}				\label{eq11.21fr}
\norm{w}_{L^q(\Omega)} \le C(n, \Omega, p) \norm{f}_{L^p(\Omega)},\quad \text{where }\; \tfrac{1}{q}=\tfrac{1}{p}-\tfrac{2}{n},
\end{equation}
and for $1<p<n$, we have
\[
\norm{\nabla w}_{L^q(\Omega)} \le C(n, \Omega, p) \norm{f}_{L^p(\Omega)},\quad \text{where }\; \tfrac{1}{q}=\tfrac{1}{p}-\tfrac{1}{n}.
\]
In particular, for $\frac{n}{2}<p<n$, we have by Morrey's theorem that
\begin{equation}				\label{eq21.23th}
[w]_{C^{0,\mu}(\Omega)} \le C(n, \Omega, p) \norm{f}_{L^p(\Omega)},\quad \text{where }\; \mu=1-\tfrac{n}{q}=2-\tfrac{n}{p}.
\end{equation}
Hereafter, we set
\[
\mathbf{g}:= w \mathbf{I}.
\]
Note that by \eqref{eq20.40th} and \eqref{eq21.08th}, $u \in L^2(\Omega)$ is an adjoint solution of
\begin{equation}				\label{eq21.48th}
L^\ast u = \dv^2 \mathbf{g}\;\text{ in }\;\Omega,\quad u=0 \;\text{ on }\;\partial\Omega.
\end{equation}
By Lemma~\ref{lem01} and \eqref{eq11.21fr}, we see that $u \in L^q(\Omega)$ for  $q\in (\frac{n}{n-2},\infty)$ and that it satisfies
\begin{equation}			\label{em1.9}
\norm{u}_{L^q(\Omega)} \le C \norm{f}_{L^{nq/(n+2q)}(\Omega)},\quad\text{where }\;C=C(n, q, \lambda, \Lambda, \Omega, \omega_{\mathbf A}).
\end{equation}
Also, by \eqref{eq21.23th}, we see that $\mathbf g$ is of Dini mean oscillation in $\Omega$ with
\[
\omega_{\mathbf g}(t) \le C\norm{f}_{L^p(\Omega)}\, t^{2-\frac{n}{p}}.
\]
Therefore, by Theorem~1.8 of \cite{DEK17}, we see that $u \in C(\overline \Omega)$.
As a matter of fact, Lemma~2.27 of \cite{DEK17} and Theorem~1.10 of \cite{DK17} with a scaling argument ($ x \mapsto rx$) reveals that for any $x_0 \in \Omega$ and $0<r< \diam \Omega$, we have
\[
\sup_{\Omega(x_0,\frac12 r)} \abs{u} \le C \left( r^{-n} \norm{u}_{L^1(\Omega(x_0, r))} + r^{2-\frac{n}{p}}\norm{f}_{L^p(\Omega)} \right),
\]
where $C=C(n,p, \lambda, \Lambda, \Omega, \omega_{\mathbf A})$.
In particular, if $f$ is supported in $\Omega(y, r)$, then by \eqref{em1.9} and H\"older's inequality, we have
\begin{equation}				\label{eq10.32fr}
\sup_{\Omega(y,\frac12 r)} \abs{u} \le C \left( r^{-\frac{n}{q}} \norm{f}_{L^{nq/(n+2q)}(\Omega(y, r))} + r^{2-\frac{n}{p}}\norm{f}_{L^p(\Omega(y, r))} \right) \le C r^2 \norm{f}_{L^\infty(\Omega(y,r))}.
\end{equation}

Therefore, if $f$ is supported in $\Omega(y, r)$, then it follows from \eqref{eq20.56th} and \eqref{eq10.32fr} that
\[
\Abs{\int_{\Omega(y, r)} f v_\epsilon \,} \le C r^2 \norm{f}_{L^\infty(\Omega(y,r))}, \quad \forall\, \epsilon \in (0, \tfrac12 r).
\]
By duality, we obtain
\begin{equation}			\label{eq1712th}
\norm{v_\epsilon}_{L^1(\Omega(y, r))} \le C r^2,\quad \forall\, \epsilon \in (0, \tfrac12 r),\quad \forall\, r \in (0, \diam \Omega),
\end{equation}
where $C=C(n, \lambda, \Lambda, \Omega, \omega_{\mathbf A})$.
We define the approximate Green's function
\[
G_\epsilon(x,y)= v_{\epsilon, y}(x)=v_\epsilon(x).
\]

\begin{lemma}				\label{lem2.15}
Let $x, y \in \Omega$ with $x \neq y$.
Then
\begin{equation}				\label{eq1723w}
\abs{G_\epsilon(x,y)} \le C \abs{x-y}^{2-n}, \quad \forall  \epsilon \in(0, \tfrac13 \abs{x-y}),
\end{equation}
where $C=C(n, \lambda, \Lambda, \Omega, \omega_{\mathbf A})$.
\end{lemma}
\begin{proof}
Let $r=\frac23 \abs{x-y}$.
If $\epsilon < \frac12 r$, then $v_\epsilon$ satisfies $L v_\epsilon =0$ in $\Omega(x,r)$.
Therefore, by the standard elliptic estimate (see \cite[Theorem 9.26]{GT}) and \eqref{eq1712th}, we have
\[
\abs{v_\epsilon(x)} \le C r^{-n} \norm{v_\epsilon}_{L^1(\Omega(x,r))} \le C r^{-n} \norm{v_\epsilon}_{L^1(\Omega(y,3r))} \le C r^{2-n}.\qedhere
\]
\end{proof}

\begin{lemma}				\label{lem2.17}
For any $y \in \Omega$ and $0<\epsilon<\diam \Omega$, we have
\begin{align}
						\label{eq1709w}
\int_{\Omega\setminus \overline B(y,r)} \abs{G_\epsilon(x,y)}^{\frac{2n}{n-2}} \,dx &\le C r^{-n}, \quad \forall \, r>0,
 \\
						\label{eq1420th}
\int_{\Omega\setminus \overline B(y,r)} \abs{D_x^2 G_\epsilon(x,y)}^2 \,dx &\le C r^{-n}, \quad \forall \, r>0,
\end{align}
where $C=C(n, \lambda, \Lambda, \Omega, \omega_{\mathbf A})$.
\end{lemma}

\begin{proof}
We first establish \eqref{eq1709w}.
In the case when $r>3\epsilon$, we get from \eqref{eq1723w} that
\[
\int_{\Omega\setminus \overline B(y,r)} \abs{G_\varepsilon(x,y)}^{\frac{2n}{n-2}}\,dx \le C \int_{\Omega \setminus \overline B(y,r)} \abs{x-y}^{-2n}\,dx \le C r^{-n}.
\]
In the case when $r\le 3\epsilon$, by \eqref{eq1659th} with $p=\frac{2n}{n+2}$ and the Sobolev's inequality, we have
\begin{equation}				\label{eq1417m}
\norm{v_\epsilon}_{L^{\frac{2n}{n-2}}(\Omega)} \le C \norm{v_\epsilon}_{W^{2,\frac{2n}{n+2}}(\Omega)}  \le C \epsilon^{1-\frac{n}{2}} \le C r ^{1-\frac{n}{2}},
\end{equation}
and thus we still get \eqref{eq1709w}.

Next, we turn to the proof of \eqref{eq1420th}.
It is enough to consider the case when $r>2\epsilon$.
Indeed, by \eqref{eq1659th}, we have
\begin{equation*}				
\int_{\Omega \setminus \overline B(y,r)} \abs{D^2 v_\epsilon}^2 \le  \int_\Omega \,\abs{D^2 v_\epsilon}^2 \le  C \epsilon^{-n} \le C r^{-n}.
\end{equation*}
For $\mathbf{g} \in C^\infty_c(\Omega\setminus \overline B(y,r))$, let $u \in L^2(\Omega)$ be an adjoint solution of \eqref{eq21.48th} so that we have
\begin{equation}				\label{eq1536fr}
\fint_{\Omega(y,\epsilon)} u = \int_\Omega \tr(\mathbf{g} D^2 v_\epsilon).
\end{equation}
Since $\mathbf{g} = 0$ in $\Omega(y,r)$, we see that  $u$ is continuous on $\overline \Omega(y, \frac12 r)$ by \cite[Theorem~1.8]{DEK17}.
In fact, it follows from \cite[Lemma~2.27]{DEK17} that
\begin{equation}			\label{eq1720m}
\sup_{\Omega(y, \frac12 r)} \, \abs{u} \le C r^{-n} \norm{u}_{L^1(\Omega(y,r))}.
\end{equation}
Therefore, by H\"older's inequality and Lemma~\ref{lem01}, we have
\[
\sup_{\Omega(y, \frac12 r)} \, \abs{u} \le C r^{-\frac{n}{2}} \norm{u}_{L^2(\Omega(y,r))} \le  r^{-\frac{n}{2}} \norm{\mathbf g}_{L^2(\Omega)}.
\]
Since $\mathbf g$ is supported in $\Omega \setminus \overline B_r(y)$, by \eqref{eq1536fr} and the above estimate, we have
\[
\Abs{\int_{\Omega \setminus \overline B(y,r)} \tr(\mathbf{g} D^2 v_\epsilon)} \le C r^{-\frac{n}{2}} \norm{\mathbf g}_{L^2(\Omega\setminus \overline B(y,r))}.
\]
Therefore, \eqref{eq1420th} follows by duality.
\end{proof}

\begin{lemma}					
For any $y \in \Omega$ and $0<\epsilon<\diam \Omega$, we have
\begin{align}
						\label{eq1151th}
\abs{\set{x \in \Omega : \abs{G_\epsilon(x,y)}>t}} &\le C t^{-\frac{n}{n-2}}, \quad \forall \, t>0, \\
						\label{eq1347th}
\abs{\set{x \in \Omega : \abs{D^2_x  G_\epsilon(x,y)}>t}} &\le C t^{-1},  \quad \forall \, t>0,
\end{align}
where $C=C(n, \lambda, \Lambda, \Omega, \omega_{\mathbf A})$.
\end{lemma}
\begin{proof}
We first establish \eqref{eq1151th}.
Let
\[
A_t=\set{x \in \Omega:\abs{G_\epsilon(x,y)}>t}
\]
and take $r=t^{-\frac{1}{n-2}}$.
Then, by \eqref{eq1709w}, we get
\[
\abs{A_t\setminus \overline B(y,r)}\le t^{-\frac{2n}{n-2}}
\int_{A_t\setminus\overline B(y,r)} \abs{G_\epsilon(x,y)}^{\frac{2n}{n-2}}\,dx \le Ct^{-\frac{2n}{n-2}}\,t^{\frac{n}{n-2}}=Ct^{-\frac{n}{n-2}}.
\]
Since $\abs{A_t\cap \overline B(y,r)}\le C r^n =Ct^{-\frac{n}{n-2}}$, obviously we thus obtain \eqref{eq1151th}.

Next, we prove \eqref{eq1347th}.
Let
\[
A_t=\set{x \in \Omega:\abs{D^2_x G_\epsilon(x,y)}>t}
\]
and take $r=t^{-\frac{1}{n}}$.
Then, by \eqref{eq1420th}, we have
\[
\abs{A_t\setminus \overline B(y,r)}\le t^{-2}
\int_{A_t\setminus\overline B(y,r)} \abs{D_x^2 G_\epsilon(x,y)}^{2}\,dx \le Ct^{-2}\,t=Ct^{-1}.
\]
Since $\abs{A_t\cap \overline B(y,r)}\le C r^n =Ct^{-1}$, we get \eqref{eq1347th}.
\end{proof}

We are now ready to construct a Green's function.
By Lemma~\ref{lem2.17}, for any $r>0$, we have
\[
\sup_{0<\epsilon <\diam\Omega} \norm{G_\epsilon(\cdot, y)}_{W^{2,2}(\Omega \setminus \overline B(y,r))} <+\infty.
\]
Therefore, by applying a diagonalization process, we see that there exists a sequence of positive numbers $\set{\epsilon_i}_{i=1}^\infty$ with $\lim_{i\to \infty} \epsilon_i=0$ and a function $G(\cdot, y)$, which belongs to $W^{2,2}(\Omega \setminus \overline B(y,r))$ for any $r>0$, such that
\begin{equation}				\label{eq1640m}
G_{\epsilon_i}(\cdot, y) \rightharpoonup G(\cdot, y) \;\text{ weakly in } W^{2,2}(\Omega \setminus \overline B(y,r)),\quad \forall\, r>0.
\end{equation}
Note that by \eqref{eq1420th}, we have
\begin{equation}				\label{eq1339fr}
\int_{\Omega\setminus \overline B(y,r)} \abs{D_x^2 G(x,y)}^2 \,dy \le C r^{-n},\quad \forall\, r>0.
\end{equation}
By compactness of the embedding $W^{2,2}  \hookrightarrow L^{\frac{2n}{n-2}}$, we also get from \eqref{eq1709w} that
\begin{equation*}				
\int_{\Omega\setminus \overline B(y,r)} \abs{G(x,y)}^{\frac{2n}{n-2}} \,dx \le C r^{-n},
\quad \forall\, r>0,
\end{equation*}
On the other hand,  \eqref{eq1151th} implies that for $1<p<\frac{n}{n-2}$, we have (see e.g., Section~3.5 in \cite{HK07})
\[
\sup_{0<\epsilon <\diam\Omega} \norm{G_\epsilon(\cdot, y)}_{L^p(\Omega)} <+\infty.
\]
Therefore, by passing to a subsequence if necessary, we see that
\[
G_{\epsilon_i}(\cdot, y) \rightharpoonup G(\cdot, y)\;\text{ weakly in } \; L^p(\Omega),\quad \forall\, p \in (1, \tfrac{n}{n-2}).
\]

For $f \in L^q(\Omega)$ with $q>\frac{n}{2}$, let $u$ be the unique adjoint solution in $L^q(\Omega)$ of the problem \eqref{eq20.40th}.
Then by \eqref{eq20.56th}, we have
\begin{equation*}
\fint_{\Omega(y,\epsilon_i)} u(x)\,dx = \int_\Omega G_{\epsilon_i}(x,y)f(x)\,dx.
\end{equation*}
By taking the limit, we get the representation formula \eqref{eq1747m}, which yields the uniqueness of the Green's function.

Finally, from \eqref{eq1640m} and \eqref{eq1655th}, we find that $G(\cdot, y)$ belongs to $W^{2,2}(\Omega\setminus \overline B(y,r))$ and satisfies $L G(\cdot, y) =0$ in $\Omega\setminus \overline B(y,r)$ for all $r>0$.
Since $\mathbf A$ is uniformly continuous in $\Omega$, by the standard $L^{p}$ theory {(see e.g., \cite{GT}), we then see that $G(\cdot, y)$ is continuous in $\Omega \setminus \set{y}$.
Moreover, by the same reasoning, we see from Lemma~\ref{lem2.17} that $G_\epsilon(\cdot, y)$ is equicontinuous on $\Omega\setminus \overline B(y,r)$ for any $r>0$.
Therefore, we may assume, by passing if necessary to another subsequence, that
\begin{equation}			\label{eq10.37w}
G_{\epsilon_i}(\cdot, y) \to G(\cdot, y)\;\text{ uniformly on }\; \Omega\setminus \overline B(y,r),\quad  \forall\, r>0 .
\end{equation}
In particular, from Lemma~\ref{lem2.15}, we see that
\[
\abs{G(x,y)} \le C\abs{x-y}^{2-n}.
\]
Therefore,  we have shown \eqref{eq0529af}.

\subsection{Construction of Green's function for the adjoint operator}

To construct Green's function for the adjoint operator $L^\ast$ given in \eqref{adj_L}, we follow the same scheme in Section~\ref{SS:Green_L}.
For $y \in \Omega$ and $\epsilon>0$, consider the following adjoint problem:
\begin{equation}			\label{eq1544mo}
L^\ast u=\frac{1}{\abs{\Omega(y,\epsilon)}} \chi_{\Omega(y,\epsilon)}\;\text{ in }\;\Omega,\quad u=0\;\text{ on }\;\partial\Omega.
\end{equation}
By Lemma~\ref{lem01}, there exists a unique adjoint solution $u_\epsilon = u_{\epsilon; y}$ in $L^2(\Omega)$ such that
\begin{equation*}			
\norm{u_\epsilon}_{L^{2}(\Omega)} \le C\epsilon^{-\frac{n}{2}}, \quad \forall \,\epsilon \in (0, \diam \Omega),
\end{equation*}
where $C=C(n, \lambda, \Lambda, \Omega, \omega_{\mathbf A})$.
Next, for $f \in C^\infty_c(\Omega)$, consider the problem
\begin{equation*}				
L v = f\;\text{ in }\;\Omega,\quad v=0 \;\text{ on }\;\partial \Omega.
\end{equation*}
By the standard $L^p$ theory (see e.g., \cite{GT}) and the Sobolev's inequality, we have
\[
\norm{v}_{L^q(\Omega)} \le C \norm{f}_{L^{nq/(n+2q)}(\Omega)},
\]
where $C=C(n, \lambda, \Lambda, \Omega, q, \omega_{\mathbf A})$.
Also, we have
\[
\sup_{\Omega(x_0,\frac12 r)}  \abs{v} \le C \left( r^{-n} \norm{v}_{L^1(\Omega(x_0, r))} + r^2 \norm{f}_{L^\infty(\Omega(x_0, r))} \right),\quad
\forall \,x_0 \in\Omega, \quad \forall\, r \in (0, \diam \Omega).
\]
Therefore, if $f$ is supported in $\Omega(y, r)$, then by the above estimates and H\"older's inequality, we get
\begin{equation*}				
\sup_{\Omega(y,\frac12 r)}  \abs{v} \le C r^2 \norm{f}_{L^\infty(\Omega(y,r))}.
\end{equation*}
From the identity
\begin{equation}\label{eq1560mo}
\fint_{\Omega(y,\epsilon)} v = \int_\Omega f u_\epsilon ,
\end{equation}
we then see that (c.f. \eqref{eq1712th} above)
\begin{equation*}				
\norm{u_\epsilon}_{L^1(\Omega(y, r))} \le C r^2,\quad \forall\, \epsilon \in (0, \tfrac12 r) ,\quad \forall\, r \in (0, \diam \Omega),
\end{equation*}
where $C=C(n, \lambda, \Lambda, \Omega, \omega_{\mathbf A})$.
We define
\[
G^\ast_\epsilon(x,y)= u_{\epsilon, y}(x)=u_\epsilon(x).
\]
Then, similar to Lemma~\ref{lem2.15}, for $x, y \in \Omega$ with $x \neq y$, we have
\begin{equation}				\label{eq1700mo}
\abs{G^\ast_\epsilon(x,y)} \le C \abs{x-y}^{2-n} ,\quad \forall  \epsilon \in (0, \tfrac13 \abs{x-y}).
\end{equation}
Indeed, if we set $r=\frac23 \abs{x-y}$, then for $\epsilon < \frac12 r$, we have $L^\ast u_\epsilon =0$ in $\Omega(x,r)$.
Since $\mathbf A$ is of Dini mean oscillation, by \cite[Lemma~2.27]{DEK17}, we have
\[
\abs{u_\epsilon(x)} \le C r^{-n} \norm{u_\epsilon}_{L^1(\Omega(x,r))} \le C r^{-n} \norm{u_\epsilon}_{L^1(\Omega(y,3r))} \le C r^{2-n},
\]
which yields \eqref{eq1700mo}.

By using \eqref{eq1700mo} and following the proof of Lemma~\ref{lem2.17}, we get the following estimate, which is a counterpart of \eqref{eq1709w}.
For any $y \in \Omega$ and $0<\epsilon<\diam \Omega$, we have
\begin{equation}			\label{eq1435m}
\int_{\Omega\setminus \overline B(y,r)} \abs{G^\ast_\epsilon(x,y)}^{\frac{2n}{n-2}} \,dx \le C r^{-n},
\quad \forall\, r>0.
\end{equation}
Indeed, by taking $f=\frac{1}{\abs{\Omega(y,\epsilon)}} \chi_{\Omega(y,\epsilon)}$ and $q=\frac{2n}{n-2}$ in \eqref{em1.9}, we have
\[
\norm{u_\epsilon}_{L^{\frac{2n}{n-2}}(\Omega)} \le C \norm{f}_{L^{\frac{2n}{n+2}}(\Omega)} \le C \epsilon^{1-\frac{n}{2}},
\]
which corresponds to \eqref{eq1417m}.
Then, by the same proof of Lemma~\ref{lem2.17}, we get \eqref{eq1435m}.
By using \eqref{eq1435m} and proceeding as in the proof of  \eqref{eq1151th}, we obtain
\[
\abs{\set{x \in \Omega : \abs{G^\ast_\epsilon(x,y)}>t}} \le C t^{-\frac{n}{n-2}}, \quad \forall\, t>0,
\]
which in turn implies that for $0<p<\frac{n}{n-2}$, there exists a constant $C_p$ such that
\[
\int_\Omega \,\abs{G^\ast_\epsilon(x,y)}^p \,dx \le C_p,\quad \forall\, y \in \Omega, \;\; \forall\, \epsilon \in (0,\diam \Omega).
\]
Therefore, for any $1 < p< \frac{n}{n-2}$, we obtain
\[
\sup_{0 < \epsilon < \diam \Omega} \|G^\ast_\epsilon (\cdot, y)\|_{L^{p}(\Omega)} < + \infty,
\]
and thus, there exists a sequence of positive numbers $\{\epsilon_j\}_{j=1}^{\infty}$ with $\lim_{j\to\infty} \epsilon_j = 0$ and a function $G^{*}(\cdot, y) \in L^p(\Omega)$ such that
\begin{equation}			\label{eq2257tue}
G^\ast_{\epsilon_j}(\cdot, y) \rightharpoonup G^\ast(\cdot, y)\;\text{ weakly in } \; L^p(\Omega).
\end{equation}

For $f\in L^{q}(\Omega)$ with $q > \frac{n}{2}$, let $v \in W^{2,q}(\Omega)\cap W^{1,q}_0(\Omega)$ be the strong solution of
\[
Lv=f \;\text{ in }\;\Omega,\quad v=0\;\text{ on }\;\partial \Omega.
\]
Then, we have (c.f. \eqref{eq1560mo} above)
\[
\fint_{\Omega(y,\epsilon_{j})} v(x) \,dx = \int_\Omega  G^\ast_{\epsilon_j} (x,y) f(x)\,dx,
\]
and thus, by taking the limit, we also get the representation formula
\[
v(y)=\int_\Omega G^\ast(x,y)f(x)\,dx.
\]
Finally,  by \eqref{eq1544mo}, \eqref{eq2257tue}, and Theorem~1.10 of \cite{DK17}, we see  that $G^\ast(\cdot, y)$ is continuous away from its singularity at $y$.

\subsection{Proof of symmetry \eqref{symmetry}}

For $x \neq y$ in $\Omega$, choose two sequences $\set{\epsilon_i}_{i=1}^{\infty}$ and $\set{\delta_j }_{j=1}^{\infty}$ such that $0< \epsilon_i, \, \delta_j < \frac{1}{3} \abs{x-y}$ for all $i$, $j$ and $\epsilon_i$, $\delta_j \to 0$ as $i$, $j \to \infty$.
From the construction of $G_\epsilon(\cdot, y)$ and $G^\ast_\epsilon(\cdot, x)$, we observe the following:
\[
\fint_{\Omega(x,\delta_j)} G_{\epsilon_i}(\cdot , y)=
\int_{\Omega} G_{\epsilon_i} (\cdot, y) \, L^\ast G^\ast_{\delta_j} (\cdot, x) = \int_{\Omega} LG_{\epsilon_i} (\cdot, y) \, G^\ast_{\delta_j} (\cdot, x)
= \fint_{\Omega(y,\epsilon_i)} G^\ast_{\delta_j}(\cdot , x).
\]
By \eqref{eq2257tue}  and the continuity of $G^\ast(\cdot, x)$ away from its singularity, we get
\[
\lim_{i\to \infty}\lim_{j\to \infty} \fint_{\Omega(y,\epsilon_i)} G^\ast_{\delta_j}(\cdot , x) = \lim_{i\to \infty} \fint_{\Omega(y,\epsilon_i)} G^\ast(\cdot , x) = G^\ast(y,x).
\]
On the other hand,  by the continuity of $G_{\epsilon_i}(\cdot, y)$ and \eqref{eq10.37w}, we obtain
\[
\lim_{i\to \infty}\lim_{j\to \infty}\fint_{\Omega(x,\delta_j)} G_{\epsilon_i}(\cdot , y) = \lim_{i\to \infty}  G_{\epsilon_i}(x, y) = G(x,y).
\]
We have thus shown that
\[
G(x,y) = G^\ast (y,x), \quad  \forall x\neq y.
\]

So far, we have seen that there is a subsequence of $\set{\epsilon_i}$ tending to zero such that
$G_{\epsilon_{k_{i}}}(\cdot, y) \to G(\cdot, y)$.
However, for $x\neq y$, we have
\[
G_{\epsilon} (x,y) = \lim_{j\to \infty} \fint_{\Omega(x, \delta_j)} G_{\epsilon}(\cdot, y)
= \lim_{j\to \infty} \fint_{\Omega(y,\epsilon)} G^\ast_{\delta_j} (\cdot, x)
= \fint_{\Omega(y,\epsilon)} G^\ast (\cdot, x) =  \fint_{\Omega(y,\epsilon)} G (\cdot, y).
\]
That is, we have
\[
G_\epsilon(x,y)=\fint_{\Omega(y,\epsilon)} G(x,z) \,dz.
\]
Therefore, we find that
\[
\lim_{\epsilon \to 0} G_{\epsilon} (x,y) = G(x,y), \quad \forall x\neq y.
\]

\subsection{Proof of estimates \eqref{eq0530af} and \eqref{eq0531af}}

We now show \eqref{eq0530af} and \eqref{eq0531af}.
Let $v=G(\cdot, y)$ and $r=\frac13 \abs{x-y}$.
Note that we have
\[
L v =0\;\text{ in }\;\Omega(x,2r),\quad v=0\;\text{ on }\;\partial\Omega(x,2r).
\]
By the standard elliptic theory, we have
\[
\norm{D v}_{L^\infty(\Omega(x, r))} \le C r^{-1-n} \norm{v}_{L^1(\Omega(x, 2r))}.
\]
Therefore, by \eqref{eq0529af}, we get
\[
\abs{D v(x)} \le C r^{-1-n} \norm{v}_{L^1(\Omega(y, 4r))} \le C r^{1-n},
\]
from which \eqref{eq0530af} follows.

In the case when $\Omega$ has $C^{2, {\rm Dini}}$ boundary, by Theorem~1.5 of \cite{DEK17}, we have
\[
\sup_{\Omega(x, r)} \abs{D^2 v} \le C r^{-n} \norm{D^2 v}_{L^1(\Omega(x, 2r))}.
\]
Therefore, by H\"older's inequality and \eqref{eq1339fr}, we have
\[
\abs{D^2 v(x)}  \le C r^{-\frac{n}{2}} \norm{D^2 v}_{L^2(\Omega(x, 2r))} \le  C r^{-\frac{n}{2}} \norm{D^2 v}_{L^2(\Omega \setminus B(y,r))} \le C r^{-n},
\]
from which \eqref{eq0531af} follows.

\section{Appendix}
The following lemma is well known to experts and essentially due to Campanato.
We provide its proof for the reader's convenience.
\begin{lemma}
Let $\Omega \in \bR^n$ be a domain satisfying the following condition:
there exists a constant $A_0 \in (0,1]$ such that for every $x \in\Omega$ and $0<r<\diam \Omega$, we have
\[
\abs{\Omega(x, r)} \ge A_0 \abs{B(x,r)},\quad \text{where }\,\Omega(x,r):= \Omega \cap B(x, r).
\]
Suppose that a function $u \in L_{\rm loc}^1(\overline \Omega)$ is of Dini mean oscillation in $\Omega$, then  there exists a uniformly continuous function $u^\ast$ on $\Omega$ such that $u^\ast = u$ a.e. in $\Omega$.
\end{lemma}
\begin{proof}
In the proof we shall denote
\[
\bar u_{x,r}= \bar u_{\Omega(x,r)}=\fint_{\Omega(x,r)} u\quad\text{and}\quad\omega(r)=\omega_u(r)=\sup_{x\in \overline{\Omega}} \fint_{\Omega(x,r)} \,\abs{u(y)-\bar u_{x,r}}\,dy.
\]
By taking the average over $\Omega(x, \frac12 r)$ to the triangle inequality
\[
\abs{\bar u_{x,r} - \bar u_{x,\frac12 r}} \le \abs{u- \bar u_{x,r}} + \abs{u-\bar u_{x, \frac12 r}}
\]
and using  $\abs{\Omega(x, r)} / \abs{\Omega(x,\frac12 r)} \le  2^{n} /A_0$, we get
\[
\abs{\bar u_{x,r} - \bar u_{x,r/2}} \le (2^n/A_0) \omega(r) + \omega(\tfrac12 r) \le (2^n/A_0) \left( \omega(r)+ \omega(\tfrac12 r) \right).
\]
By telescoping, we get
\begin{equation}			\label{eq1109w}
\begin{aligned}
\abs{\bar u_{x,r}- \bar u_{x, 2^{-k}r}} &\le \sum_{j=0}^{k-1}\, \abs{\bar u_{x,2^{-j}r}-\bar u_{x, 2^{-(j+1)}r}}\\
& \le (2^n/A_0)\left( \omega(r)+2\omega(\tfrac12 r)+ \cdots + 2 \omega(\tfrac{1}{2^{k-1}} r) + \omega(\tfrac{1}{2^k} r)\right) \\
& \le (2^{n+1}/A_0) \sum_{j=0}^\infty \omega(\tfrac{1}{2^j} r)  \lesssim \int_0^r \frac{\omega(t)}{t}\,dt,
\end{aligned}
\end{equation}
where in the last step we used the fact that $\omega(t) \simeq \omega(\frac{1}{2^j}r)$ when $t \in (\frac{1}{2^{j+1}}r, \frac{1}{2^j}r]$; see \cite{DK17}.
Note that the last inequality also implies that
\begin{equation}					\label{eq1048tue}
\omega(r) \lesssim \int_0^r \frac{\omega(t)}{t}\,dt.
\end{equation}
Now, we define the function $u^\ast$ on $\Omega$ by setting $u^\ast(x)=\lim_{r\to 0} \bar u_{x,r}$.
By the Lebesgue differentiation theorem, we have $u=u^*$ a.e.
By letting $k\to \infty$ in \eqref{eq1109w}, we obtain
\begin{equation}					\label{eq1011tue}
\abs{u^*(x)-\bar u_{x,r}} \lesssim \int_0^r \frac{\omega(t)}{t}\,dt\quad\text{for a.e. $x \in \Omega$.}
\end{equation}

For any $x$, $y$ in $\Omega$, let $r=\abs{x-y}$, $z=\frac12 (x+y)$, and use \eqref{eq1011tue} to get
\[
\abs{u^\ast(x)-u^\ast(y)} \le \abs{u^\ast(x)-\bar u_{x,r}}+ \abs{u^\ast(y)-\bar u_{y,r}} + \abs{\bar u_{x,r}-\bar u_{y,r}}  \lesssim \int_0^{\abs{x-y}} \frac{\omega(t)}{t}\,dt +\abs{\bar u_{x,r}-\bar u_{y,r}}.
\]
By taking the average over $\Omega(z,\frac12 r)$ to the triangle inequality
\[
\abs{\bar u_{x,r} - \bar u_{y,r}} \le \abs{u- \bar u_{x,r}} + \abs{u-\bar u_{y, r}}
\]
and noting that $\Omega(z, \frac12 r) \subset \Omega(x,r)\cap \Omega(y,r)$, we get
\[
\abs{\bar u_{x,r} - \bar u_{y,r}} \le (2^{n+1}/A_0) \omega(\abs{x-y}).
\]
Combining together and using \eqref{eq1048tue}, we conclude that
\[
\abs{u^\ast(x)-u^\ast(y)} \lesssim \int_0^{\abs{x-y}} \frac{\omega(t)}{t}\,dt + \omega(\abs{x-y}) \lesssim \int_0^{\abs{x-y}} \frac{\omega(t)}{t}\,dt.
\]
Therefore, we see that $u^\ast$ is uniformly continuous with its modulus of continuity dominated by the function $\displaystyle\rho(r):=\int_0^r \frac{\omega(t)}{t}\,dt$.
\end{proof}


\end{document}